\documentclass[final,3p,times]{article}

\usepackage{lineno,hyperref}
\modulolinenumbers[1]

\usepackage{graphicx}
\usepackage{amssymb}
\usepackage{mathtools}
\usepackage{amsthm}
 \usepackage{amsmath}
\usepackage{dsfont}
\usepackage{epsfig}
\usepackage{float}
\usepackage{epstopdf}
\usepackage{subfigure}
\usepackage{cite}
\usepackage{xcolor}
\usepackage{latexsym,amssymb}
\usepackage{pgfplots}
\usepackage{amsmath}
\usepackage{authblk}

\numberwithin{equation}{section}

\theoremstyle{definition}
\newtheorem{exmp}{Example}
\newtheorem{theorem}{Theorem}[section]
\newtheorem{lemma}[theorem]{Lemma}
\newtheorem{proposition}[theorem]{Proposition}

\newtheorem{definition}[theorem]{Definition}

\theoremstyle{remark}

	\usetikzlibrary{plotmarks}
\usetikzlibrary{external}
\pgfplotsset{compat=1.3}
\newlength\figurewidth 
\newlength\figureheight
\tikzset{external/force remake=true}

\usepackage[margin=0.87in]{geometry}

\date{}

\begin{document}



\title{Non-polynomial divided difference and blossoming}

\author[a]{Fatma Z\"{u}rnac{\i}-Yeti\c{s}  \footnote{\textbf{Email addresses:} $^a$fzurnaci@itu.edu.tr }}

\affil[a]{Department of Mathematics Engineering, Istanbul Technical University,  Maslak, Istanbul, 34469, Turkiye}

\setcounter{Maxaffil}{0}
\renewcommand\Affilfont{\small}

\maketitle

\begin{abstract}
	Two notable examples of dual functionals in approximation theory and computer-aided geometric design are the blossom and the divided difference operator. Both of these dual functionals satisfy a similar set of formulas and identities. Moreover, the divided differences of polynomials can be expressed in terms of the blossom. In this paper, an extended non-polynomial homogeneous blossom for a wide collection of spline spaces, including trigonometric splines, hyperbolic splines, and special M\"{u}ntz spaces of splines, is defined. It is shown that there is a relation between the non-polynomial divided difference and the blossom, which is analogous to the polynomial case.
\end{abstract}
{\bf Keywords}  	Non-polynomial divided difference, Homogeneous blossom, Dual functionals
 \\
 {\bf MSC2020 Classification:} 65D17, 41A10

\section{Introduction}
Dual functionals are linear operators that vanish on all basis functions except one, thereby extracting the corresponding coefficient in the expansion of the function with respect to the chosen basis. The divided difference operator and the blossom are two of the most significant examples of dual functionals in approximation theory and computer-aided geometric design. The divided difference provides the dual functionals for the Newton basis, and the blossom represents the dual functionals for the B-splines.

The divided differences are commonly used in numerical analysis and approximation
theory, and are a basic tool in interpolation and approximation by polynomials and in spline theory. Divided differences are directly involved in the definition of B-splines and  also provide the Newton coefficients of the polynomial interpolant. Blossoming is a powerful technique for developing the theory of B\'{e}zier and B-spline curves and surfaces.  The coefficient of B-spline is formed by evaluating the polar form of the polynomial at the knots of the B-spline, and the B\'{e}zier coefficients of a polynomial curve are determined by its blossom evaluated at zeros and ones.  Although the blossom and divided difference seem like two different operators, Goldman shows that  there exists an important relationship between the blossom and the divided difference  in \cite{ron2}. These results demonstrate that the divided difference is in fact  a special case of an extended version of the blossom, and this extended blossom can be constructed directly from divided differences. 

In this work, the relationship between divided difference and blossoming is  generalized for the space	\begin{equation}
	\pi_{n}(\gamma_{1},\gamma_{2})=\text{span}\big\{\gamma_{1}^{n-k}\gamma_{2}^{k}\big\}_{k=0}^{n} ,
\end{equation}
where the functions $ \gamma_{1} $ and $ \gamma_{2} $ are linearly independent.  The space $\pi_{n}(\gamma_{1},\gamma_{2})$ includes as special cases
standard polynomials, trigonometric polynomials, hyperbolic polynomials and special
M\"{u}ntz spaces, as well as many other spaces. Our main goal is to define an  extended non-polynomial homogeneous blossom for the space   $	\pi_{n}(\gamma_{1},\gamma_{2})$  and  establish the relationship between the non-polynomial divided difference and this blossom.

This paper is organized in the following fashion. In Section \ref{section2}, we introduce the basic definitions, fundamental formulas,
and explicit notation for non-polynomial divided differences. Some basic information for blossoming is reviewed in Section
\ref{section3}. In Section \ref{section4}, we invoke the definition of the extended non-polynomial homogeneous blossom for the space $\pi_{n}$, and then the relation between the non-polynomial divided difference and blossoming is given. In Section \ref{conc}, we conclude with a short summary of our results along with a brief discussion of an open problem for future research.

\section{Non-polynomial divided differences}\label{section2}

Classical polynomial divided differences can be defined by a recurrence relation, by a ratio of Vandermonde determinants or using some integral representations. In addition to the classical divided differences, divided differences for non-polynomial spaces, such as trigonometric, hyperbolic or Chebyshevian spaces are studied \cite{schumaker1, Schumaker3, Lyche1, Lyche2}. These various types of divided differences are defined  as  a ratio of Vandermonde determinants. 
Recently, in \cite{fatma}, using the barycentric coordinates of a point on the planar parametric curve $P(t)=(\gamma_{1}(t),\gamma_{2}(t))$, $a \leq t\leq b$ satisfying
\begin{align}\label{dfunc}
	d(x_{1},x_{2})=\gamma_{1}(x_{1})\gamma_{2}(x_{2})-\gamma_{1}(x_{2})\gamma_{2}(x_{1})
\end{align}
that never vanishes for distinct $x_{1},x_{2} \in [a,b]$,  Z\"{u}rnac{\i} and Di\c{s}ib\"{u}y\"{u}k   give an explicit representation of non-polynomial B-spline functions  for the spaces $\pi_{n}(\gamma_{1},\gamma_{2})$ which are a wide collection of spline
spaces. These non-polynomial B-splines are constructed by using non-polynomial divided differences applied to a proper generalization of the truncated-power function. In \cite{fatma}, starting with the interpolation problem, non-polynomial divided differences are defined recursively as a generalization of classical divided differences by the following definition.
\begin{definition}[cf. \cite{fatma}]
	Given $n+1$ distinct abscissas $x_{j},x_{j+1},\ldots,x_{j+n}$, the non-polynomial divided difference of order $ n $ is defined recursively as
	\begin{align}\label{divdif}
		f_{\gamma_{1}, \gamma_{2}}[x_{j},\ldots,x_{j+n}]=\frac{f_{\gamma_{1}, \gamma_{2}}[x_{j},\ldots,x_{j+n-2},x_{j+n}] - f_{\gamma_{1}, \gamma_{2}}[x_{j},\ldots,x_{j+n-1}]}{d(x_{j+n-1},x_{j+n})},	
	\end{align}
	where $ f_{\gamma_{1}, \gamma_{2}}[x_{k}]=f(x_{k}) $ for $ k=j,j+1,\ldots, j+n. $
\end{definition}
To generalize the classical polynomial Hermite interpolation to the non-polynomial case, Z\"{u}rnac{\i} and Di\c{s}ib\"{u}y\"{u}k also define a new derivative operator in \cite{fatma}.
\begin{definition}[cf. \cite{fatma}]\label{derdef}
	The derivative of $ f $ at a point $ x_{0} $ is	
	\begin{equation}\label{derdef1}
		D_{\gamma_{1},\gamma_{2}}\big(f(x_{0})\big):=\lim\limits_{x\to x_{0}}\frac{f(x)-f(x_{0})}{d(x_{0},x)}
	\end{equation}
	provided that the limit exists.
\end{definition}
The notation  $D^{n}_{\gamma_{1},\gamma_{2}} f(x) $ is used when the operator $ D_{\gamma_{1},\gamma_{2}} $ is applied $ n$ times to the function $ f(x)$. Note that for $ \gamma_{1}(x) = 1$ and $\gamma_{2}(x) = x$,\, $D_{\gamma_{1},\gamma_{2}}\big(f(x_0)\big)$ reduces to the classical derivative. Hence, we may consider the operator $D_{\gamma_{1},\gamma_{2}}$ as a generalization of the classical derivative operator. Like the classical derivative, the generalized derivative satisfies  sum, difference, product and quotient rules as well as Leibniz formula expressing the derivative of nth order of the product of two functions.

Numerous identities and properties for non-polynomial divided differences are investigated in \cite{fatma3, fatma2}. The non-polynomial divided difference $f_{\gamma_{1}, \gamma_{2}}[x_{0},\ldots,x_{n}]$ of a differentiable function $f(x)$ can be characterized by a simple collection of axioms. Let $1\in  \text{span}\big\{ \gamma_{1}, \gamma_{2}\big\}$, \\
\textbf{\textit{Symmetry}}
\begin{equation*}
	f_{\gamma_{1}, \gamma_{2}}[x_{0},\ldots,x_{n}]= f_{\gamma_{1}, \gamma_{2}}[x_{\sigma(0)},\ldots,x_{\sigma(n)}]
\end{equation*}
\textbf{\textit{Cancellation}}
\begin{equation*}
	f(x)_{\gamma_{1}, \gamma_{2}}[x_{0},x_{1},\ldots,x_{n}]=\{d(x_{n+1},x)f(x)\}_{\gamma_{1}, \gamma_{2}}[x_{0},x_{1},\ldots,x_{n+1}].
\end{equation*}
\textbf{\textit{Differentiation}}
\begin{equation*}
	f_{\gamma_{1}, \gamma_{2}}[\underbrace{x_{0},x_{0},\ldots,x_{0}}_{n+1\text{ terms}}]=\frac{D^{n}_{\gamma_{1},\gamma_{2}}f(x_{0})}{n!}.
\end{equation*}
Also, the classical Taylor theorem is generalized in \cite{fatma3, fatma2}.

\begin{theorem}[Generalized Taylor Theorem]\label{generalizetaylor}
	Let $x, x_{0} \in [a,b]$ and  $f, \gamma_{1}, \gamma_{2} \in C^{n}[a,b]$  and let $f^{n+1}$, $\gamma_{1}^{n+1}$ and $\gamma_{2}^{n+1}$ exist in the open interval $(a,b)$ .  If $1\in \text{span}\big\{\gamma_{1}, \gamma_{2}\big\}$, then 
	\begin{align*}
		f(x)=  \sum_{k=0}^{n}\frac{D_{\gamma_{1}, \gamma_{2}}^{k}f(x_{0})}{k!}d(x_{0},x)^{k} + R_n (x),
	\end{align*}
	where 
	\begin{align*}
		R_n (x)=D_{\gamma_{1}, \gamma_{2}}^{n+1}f(\xi_{x})\frac{(d(x_{0},x))^{n+1}}{(n+1)!},\,\,\, \xi_{x}\in (a,b).
	\end{align*}	
\end{theorem}
Using the properties of the generalized derivative, we can easily derive that the generalized Taylor formula for a function in the space $ \pi_{n}(\gamma_{1},\gamma_{2})$ is equal to the function itself.

\section{Blossoming}\label{section3}

In this section, some basic information concerning blossoming is  briefly reviewed.   The blossom of a degree $n$ polynomial $P(x)$ is the function $p(x_1, \ldots, x_n)$ that satisfies the following three axioms: 
\begin{itemize}
	\item[]\textit{Symmetry}
	\begin{equation*}
		p(x_1, \ldots, x_n)= p(x_{\sigma(1)}, \ldots, x_{\sigma(n)})
	\end{equation*}
	\item[] \textit{Multiaffine}
	\begin{equation*}
		p(x_1,\ldots,(1-\alpha)u+\alpha w,\ldots, x_n)=(1-\alpha)p(x_1,\ldots,u,\ldots, x_n)+\alpha p(x_1,\ldots, w,\ldots, x_n)
	\end{equation*}
	\item[]\textit{Diagonal}
	\begin{equation*}
		p(\underbrace{x, \ldots, x}_n)= P(x)
	\end{equation*}
\end{itemize}

There is also a multilinear version of the blossom for homogeneous polynomials. The blossom   of a degree $n$ homogeneous polynomial $P(x,w)$ is the function $p((x_1,w_1),\ldots, (x_n,w_n))$ that satisfies the following three axioms:
\begin{itemize}
	\item[] \textit{Symmetry}
	\begin{equation*}
		p((x_1,w_1), \ldots, (x_n,w_n))= p((x_{\sigma(1)},w_{\sigma(1)}), \ldots, (x_{\sigma(n)},w_{\sigma(n)}))
	\end{equation*}
	\item[] \textit{Multilinear}
	\begin{align*}
		p((x_1,w_1), \ldots, \alpha (x_{i},w_{i})+\beta( x_{i},v_i), \ldots, (x_n,w_n))=\alpha p((x_1,w_1), \ldots,  (x_{i},w_{i}), \ldots, (x_n,w_n))\\+\beta p((x_1,w_1), \ldots,( x_{i},v_i), \ldots, (x_n,w_n))
	\end{align*}
	\item[] \textit{Diagonal}
	\begin{equation*}
		p(\underbrace{(x,w), \ldots, (x,w)}_n)=P(x,w).
	\end{equation*}
\end{itemize}
We may retrieve the univariate polynomial $P(x, 1)$  from the homogeneous polar form of $P(x,w)$ by setting  $(x_i ,w_i ) = (x, 1)$, $i = 1,\ldots,n $. Thus
$$P(x, 1) = p((x, 1), \ldots , (x, 1), (x, 1)).$$

In \cite{ron1}, Goldman defines the extended blossom such that the standard blossom is extended to incorporate additional parameters.

\begin{definition} The extended blossom of order $k \in \mathbb{Z}$ of a function $F(x)$ is a function 	$f(x_{1},\ldots,x_{n+k}/v_{1},\ldots,v_{k})$ which satisfies the following axioms:
	\begin{itemize}
		\item[] \textit{Bisymmetry}
		\begin{equation*}
			f(x_{1},\ldots,x_{m} / v_{1},\ldots,v_{n})= 	f(x_{\sigma(1)},\ldots,x_{\sigma(m)}/v_{\tau(1)},\ldots,v_{\tau(n)})
		\end{equation*}
		\item[] \textit{Multiaffine in  $x$}
		\begin{align*}
			f(x_1, \ldots, (1-\alpha) u+\alpha v, \ldots, x_{m}/v_{1},\ldots,v_{n})=(1-\alpha)	f(x_1, \ldots,u, \ldots, x_{m}/v_{1},\ldots,v_{n})\\+\alpha	f(x_1, \ldots, v, \ldots, x_{m}/v_{1},\ldots,v_{n})
		\end{align*}
		\item[] \textit{Cancellation 
		}
		\begin{align*}
			f(x_{1},\ldots,x_{m},w/v_{1},\ldots,v_{n},w)=	f(x_{1},\ldots,x_{m}/v_{1},\ldots,v_{n})
		\end{align*}
		\item[] \textit{Diagonal}
		\begin{equation*}
			f(\underbrace{x, \ldots, x}_{m}/\underbrace{x, \ldots,x}_{n})=F(x).
		\end{equation*}
	\end{itemize}
\end{definition}
When $k=m-n\geq 0$, the function $F(x)$ must be a polynomial in $x$ of degree less than or equal to $k$. In \cite{ron1}, Goldman establishes the existence and uniqueness of the extended blossom of $P(x)$. When $k=m-n < 0$, the function $F(x)$ need no longer be a polynomial in $x$,  $F(x)$ may be either differentiable or analytic functions. In \cite{ron1}, Goldman establishes the existence for the negative order blossom, and the uniqueness  is demonstrated in \cite{tuncer}.

In addition to polynomial spaces, the blossom of non-polynomial functions is also studied. 
Gonsor and Neamtu develop blossoms for trigonometric polynomials \cite{gonsor1}, and demonstrate that the dual functionals for trigonometric Bernstein basis functions are given by the blossom for trigonometric polynomials. In \cite{cetin2}, Di\c{s}ib\"{u}y\"{u}k and Goldman generalize these results to more general spaces $	\pi_{n}(\gamma_{1},\gamma_{2})$   introducing a more general notion of blossoms, and the homogeneous blossoming theory is generalized for the space $	\pi_{n}(\gamma_{1},\gamma_{2})$. To\, construct blossoms for non-polynomial spaces, one needs to change at least one of the three polynomial blossoming axioms. Two axioms are flexible, but one usually retains the symmetry property. Gonsor and Neamtu \cite{gonsor1} have examined polar forms for trigonometric polynomials. They begin with the multiaffine polar form, maintain its symmetry and diagonal properties, and substitute a multi-barycentric property for the multiaffine one. In  \cite{cetin2},  Di\c{s}ib\"{u}y\"{u}k and Goldman start with the multilinear variant of the polar form for homogeneous polynomials, maintain the multilinear and symmetry properties, and modify the diagonal property by substituting the values  $ \gamma_{1}(x) $ and $ \gamma_{2}(x) $ for the parameter pairs $ (x,w)$. They give the following definition for the non-polynomial homogeneous blossom:

\begin{definition}
	Let	$G\in \pi_{n}(\gamma_{1},\gamma_{2}) $. Then $g$ is called the blossom of $G$ if it satisfies the following three axioms:
	\begin{itemize}
		\item[]\textit{Symmetry}
		\begin{equation*}
			g((x_1,w_1), \ldots, (x_n,w_n))= g((x_{\sigma(1)},w_{\sigma(1)}), \ldots, (x_{\sigma(n)},w_{\sigma(n)}))
		\end{equation*}
		\item[]\textit{Multilinear}
		\begin{align*}
			g((x_1,w_1), \ldots, \alpha (x_{i},w_{i})+\beta( x_{i},v_i), \ldots, (x_n,w_n))=\alpha g((x_1,w_1), \ldots,  (x_{i},w_{i}), \ldots, (x_n,w_n))\\+\beta g((x_1,w_1), \ldots,( x_{i},v_i), \ldots, (x_n,w_n))
		\end{align*}
		\item[]\textit{Diagonal}
		\begin{equation*}
			g((\gamma_{1}(x),\gamma_{2}(x)), \ldots, (\gamma_{1}(x),\gamma_{2}(x)))=G(x)
		\end{equation*}
	\end{itemize}
\end{definition}

\section{Non-polynomial divided difference and blossoming}\label{section4}
The non-polynomial homogeneous blossom given in \cite{cetin2} exists. Now we can give a constructive proof for existence.
\begin{theorem}\label{standart_b}
	Let $G\in \pi_{k}(\gamma_{1},\gamma_{2}) $ and $k\leq m$ . If $1\in \text{span}\big\{\gamma_{1}, \gamma_{2}\big\}$, then for all $\tau$	
	\begin{align}
		g((x_{1},w_{1}),\ldots,(x_{m},w_{m}))=\sum_{j}\frac{(-1)^{m-j}}{m!}D_{\gamma_{1}, \gamma_{2}}^{j}\Psi(\tau) D_{\gamma_{1}, \gamma_{2}}^{m-j}G(\tau),
	\end{align}
	where
	\begin{align} 
		\Psi(x)=\prod_{i=1}^{m}\left(x_i \gamma_{2}(x)-w_i \gamma_{1}(x)\right).
	\end{align}
\end{theorem}

\begin{proof}
	We shall show that the axioms of the non-polynomial homogeneous blossom are satisfied. It is clear from the function $\Psi(x)$ that $g((x_{1},w_{1}),\ldots,(x_{m},w_{m}))$ is symmetric and multilinear. It is possible to deduce the diagonal property by noting that the right-hand side becomes the generalized Taylor expansion of $G(t)$ at $\tau$ when $(x_{1},w_{1})=\ldots=(x_{m},w_{m})= (\gamma_{1}(t),\gamma_{2}(t))$.
\end{proof}

The standard extended blossom is defined by Goldman \cite{ron1,ron2}. Here the extended non-polynomial homogeneous blossom is defined.
\begin{definition} The extended  non-polynomial  homogeneous blossom of order $k \in \mathbb{Z}$ of a function $H(x)$ is a function 	$h((x_{1},w_{1}),\ldots,(x_{m},w_{m})/(u_{1},v_{1}),\ldots,(u_{n},v_{n}))$ with $k=m-n$ which satisfies the following axioms:
	\begin{itemize}
		\item[] \textit{Bisymmetry in the $(x,w)$ and $(u,v)$ parameters:}
		\begin{align*}
			h((x_{1},w_{1}),\ldots,(x_{m},w_{m})/(u_{1},v_{1})&,\ldots,(u_{n},v_{n}))\\&= h((x_{\sigma(1)},w_{\sigma(1)}),\ldots,(x_{\sigma(m)},w_{\sigma(m)})/(u_{\tau(1)},v_{\tau(1)}),\ldots,(u_{\tau(n)},v_{\tau(n)}))	
		\end{align*}
		\item[] \textit{ Multilinearity in the $(x,w)$ parameters:}
		\begin{align*}
			h((x_{1},w_{1}),\ldots, a(x_{i},w_{i})+ b(y_{i},z_{i}) ,\ldots ,&(x_{m},w_{m})/(u_{1},v_{1}),\ldots,(u_{n},v_{n}))\\=&a h((x_{1},w_{1}),\ldots,(x_{i},w_{i}),\ldots ,(x_{m},w_{m})/(u_{1},v_{1}),\ldots,(u_{n},v_{n}))\\+ & b h((x_{1},w_{1}),\ldots,(y_{i},z_{i}) ,\ldots ,(x_{m},w_{m})/(u_{1},v_{1}),\ldots,(u_{n},v_{n}))	\end{align*}
		\item[]  \textit{Cancellation}
		\begin{align*}
			h((x_{1},w_{1}),\ldots,(x_{m},w_{m}),(y,z)/(u_{1},v_{1}),\ldots,(u_{n},v_{n}),(y,z))=h((x_{1},w_{1}),\ldots,(x_{m},w_{m})/(u_{1},v_{1}),\ldots,(u_{n},v_{n}))
		\end{align*}
		\item[]  \textit{Diagonal}
		\begin{equation*}
			h(\underbrace{(\gamma_{1}(x),\gamma_{2}(x)), \ldots, (\gamma_{1}(x),\gamma_{2}(x))}_{m}/\underbrace{(\gamma_{1}(x),\gamma_{2}(x)), \ldots, (\gamma_{1}(x),\gamma_{2}(x))}_{n})=H(x)
		\end{equation*}
	\end{itemize} 
\end{definition}

Below we will establish the existence and uniqueness of this extended homogeneous blossom of a function from the space $\pi_{n}$ for $m \geq n$. However, before doing so, we  examine a few simple examples.
\begin{exmp}Let $G\in \pi_{l}(\gamma_{1},\gamma_{2}), \,l\leq k$, and let $g_{e}((x_{1},w_{1}),\ldots,(x_{n+k},w_{n+k})/(u_{1},v_{1}),\ldots,(u_{n},v_{n}))$	denote the extended  non-polynomial  homogeneous blossom of $G(x)$ of order $k$.
	\begin{itemize}
		\item[1.] $G(x)=1$	\\
		$$	g_{e}((x_{1},w_{1}),\ldots,(x_{n+k},w_{n+k})/(u_{1},v_{1}),\ldots,(u_{n},v_{n}))=1$$
		\item[2.]  $G(x)=d(a,x)$
		$$	g_{e}((x_{1},w_{1}),\ldots,(x_{n+k},w_{n+k})/(u_{1},v_{1}),\ldots,(u_{n},v_{n}))=\frac{\sum_{i=1}^{n+k}(x_{i}\gamma_{2}(a)-w_{i}\gamma_{1}(a))-\sum_{i=1}^{k}(u_{i}\gamma_{2}(a)-v_{i}\gamma_{1}(a))}{n}$$\\
		
		\item [3.] $G(x)=(d(a,x))^{2}$
		\begin{align*}
			g_{e}((x_{1},w_{1}),\ldots,(x_{n+k},w_{n+k})&/(u_{1},v_{1}),\ldots,(u_{n},v_{n}))\\&=\frac{1}{\binom{n}{2}}\bigg(\sum_{i<j}\left(x_{i}x_{j}(\gamma_{2}(a))^{2}-(x_{i}w_{j}+x_{j}w_{i})\gamma_{2}(a)\gamma_{1}(a)+w_{i}w_{j}(\gamma_{1}(a))^{2}\right)\\&-\sum_{i,j}\left(u_{i}x_{j}(\gamma_{2}(a))^{2}-(u_{i}w_{j}+x_{j}v_{i})\gamma_{2}(a)\gamma_{1}(a)+v_{i}w_{j}(\gamma_{1}(a))^{2}\right) \\ &+ \sum_{i\leq j}\left(u_{i}u_{j}(\gamma_{2}(a))^{2}-(u_{i}v_{j}+u_{j}v_{i})\gamma_{2}(a)\gamma_{1}(a)+ v_{i}v_{j}(\gamma_{1}(a))^{2}\right)\bigg).
		\end{align*}	
	\end{itemize}
	
\end{exmp}

\begin{theorem}\label{positive_eb}
	Let  $G\in \pi_{l}(\gamma_{1},\gamma_{2}), \,l\leq k$, and let $g((x_{1},w_{1}),\ldots,(x_{m},w_{m}))$	denote the non-polynomial homogeneous blossom of $G(x)$. Then the extended  non-polynomial  homogeneous blossom of $G(x)$ of order $k$ is given by
	\begin{align}
		g_{e}((x_{1},w_{1}),\ldots,(x_{m},w_{m})/(u_{1},v_{1}),\ldots,(u_{n},v_{n}))=\sum (-1)^{\beta} g(x_{i_{1}},w_{i_{1}}),\ldots,(x_{i_{\alpha}},w_{i_{\alpha}}),(u_{j_{1}},v_{j_{1}}),\ldots,(u_{j_{\beta}},v_{j_{\beta}})),
	\end{align}
	where the sum taken over all collections of indices ${i_{1},\ldots,i_{\alpha}}$ and ${j_{1},\ldots,j_{\beta}}$ such that \\
	(i) $i_{1},\ldots,i_{\alpha}$ are distinct\\
	(ii) $ j_{1},\ldots,j_{\beta}$ need not be distinct \\
	(iii) $\alpha+\beta=k=m-n$.
\end{theorem}

\begin{proof}
	It is necessary to verify that $g_{e}$ satisfies the axioms of the extended non-polynomial homogeneous blossom of order $k$. Using the axioms of the standard homogeneous blossom, it is clear that  $g_{e}$ is   bisymmetric and multilinear in $(x,w)$. Assume, without loss of generality, that $(x_{1},w_{1})=(u_{1},v_{1})$. Then, by symmetry,
	\begin{align*}
		(-1)^{\beta}g((x_{1},w_{1}),(x_{i_{2}},w_{i_{2}}),\ldots, (x_{i_{\alpha}},w_{i_{\alpha}})&,(u_{j_{1}},v_{j_{1}}),\ldots,(u_{j_{\beta}},v_{j_{\beta}}))\\&+(-1)^{\beta+1}g((x_{i_{2}},w_{i_{2}}),\ldots, (x_{i_{\alpha}},w_{i_{\alpha}}),(u_{j_{1}},v_{j_{1}}),\ldots,(u_{j_{\beta}},v_{j_{\beta}})=0.
	\end{align*}
	Therefore any terms including $(x_{1},w_{1})$ or $(u_{1},v_{1})$ cancel. Everything else adds up to $$	g_{e}((x_{1},w_{1}),\ldots,(x_{m},w_{m})/(u_{1},v_{1}),\ldots,(u_{n},v_{n})),$$ so $g_{e}$ satisfies the cancellation property.
	
	Lastly, the cancellation property converts $g_{e}$ to $G$ along the diagonal, that is
	\begin{align*}
		g_{e}(\underbrace{(\gamma_{1}(x),\gamma_{2}(x)),\ldots,(\gamma_{1}(x),\gamma_{2}(x))}_{m}/\underbrace{(\gamma_{1}(x),\gamma_{2}(x)),\ldots,(\gamma_{1}(x),\gamma_{2}(x))}_{n})&=g_{e}(\underbrace{(\gamma_{1}(x),\gamma_{2}(x)),\ldots,(\gamma_{1}(x),\gamma_{2}(x))}_{k}/)\\&=g(\underbrace{(\gamma_{1}(x),\gamma_{2}(x)),\ldots,(\gamma_{1}(x),\gamma_{2}(x))}_{k}/)\\&=G(x).
	\end{align*}	
\end{proof}

\begin{lemma}\label{uniq}
	Let $l\leq k=m-n$ and  $	g_{e}\left((x_{1},w_{1}),\ldots,(x_{m},w_{m})/(u_{1},v_{1}),\ldots,(u_{n},v_{n})\right)$ be the blossom of a function in $\pi_{l}(\gamma_{1},\gamma_{2})$
	that is bisymmetric in the $(x,w)$ and $(u,v)$ parameters, multilinear in the $(x,w)$ parameters, satisfies
	the cancellation property, and reduces to zero along the  $(\gamma_{1},\gamma_{2})$ diagonal. Then $g_e$ is the zero function.
\end{lemma}

\begin{proof}
	If $k=m-n$, we can rewrite  $g_{e}((x_{1},w_{1}),\ldots,(x_{k+n},w_{k+n})/(u_{1},v_{1}),\ldots,(u_{k},v_{k}))$. We proceed by induction on $k$. When $k=0$, the function $g_{e}((x_{1},w_{1}),\ldots,(x_{n},w_{n})/ )$
	is symmetric and multilinear; hence $g_{e}((x_{1},w_{1}),\ldots,(x_{n},w_{n})/ )$ is the blossom of $g_e((\gamma_{1}(x),\gamma_{2}(x)),\ldots,(\gamma_{1}(x),\gamma_{2}(x))/)$.
	However, based on the assumption, the function $g_e((\gamma_{1}(x),\gamma_{2}(x))\ldots,(\gamma_{1}(x),\gamma_{2}(x))/ )$ is equal to zero. 	Therefore, due to the uniqueness of the multilinear blossom, $g_{e}((x_{1},w_{1}),\ldots,(x_{n},w_{n})/ )$  must identically  be equal to zero. Assume that we have determined that $g_{e}((x_{1},w_{1}),\ldots,(x_{k+n},w_{k+n})/(u_{1},v_{1}),\ldots,(u_{k},v_{k}))$ is zero. Now, let's examine\\ $g_{e}((x_{1},w_{1}),\ldots,(x_{k+n+1},w_{k+n+1})/(u_{1},v_{1}),\ldots,(u_{k+1},v_{k+1}))$. Using the cancellation property and the induction hypothesis, we can conclude that the function $g_{e}((x_{1},w_{1}),\ldots,(x_{k+n+1},w_{k+n+1})/(u_{1},v_{1}),\ldots,(u_{k+1},v_{k+1}))$ becomes zero whenever $(u_{k+1},v_{k+1})$ coincides with any of the pairs $(x_i,w_i)$ since in that case it reduces to a blossom of order $k$, which is zero by the induction hypothesis. Fix all variables except $(u_{k+1},v_{k+1})$ and view the above expression as a polynomial in this single pair. Because $g_e$ is the extended blossom of a function in $\pi_l(\gamma_1,\gamma_2)$ with $l\le k$, the degree of this
polynomial is at most $k$. However, it vanishes at $n+k+1>k$ distinct values of
$(u_{k+1},v_{k+1})$, and therefore it must be identically zero. This completes
the induction.
\end{proof}

\begin{theorem}
	Let  $1\in \text{span}\big\{ \gamma_{1}, \gamma_{2}\big\}$ and $G\in \pi_{l}(\gamma_{1},\gamma_{2}), \quad l\leq k=m-n $. Then there exists a unique function $g_{e}\left((x_{1},w_{1}),\ldots,(x_{m},w_{m})/(u_{1},v_{1}),\ldots,(u_{n},v_{n})\right)$ that is bisymmetric in the $(x,w)$ and $(u,v)$ parameters, multilinear in the $(x,w)$ parameters, satisfies the cancellation property, and reduces to $G(x)$ along the  $(\gamma_{1},\gamma_{2})$  diagonal. 
\end{theorem}
\begin{proof}
	Clearly from Theorem \ref{positive_eb},   $g_{e}\left((x_{1},w_{1}),\ldots,(x_{m},w_{m})/(u_{1},v_{1}),\ldots,(u_{n},v_{n})\right)$ exists. For the uniqueness,  assume that $g_e$ and $h_e$ are two polynomials of the same degree in the functions $\gamma_{1},\gamma_{2}$
	that are bisymmetric in the $(x,w)$ and $(u,v)$ parameters, multilinear in the $(x,w)$ parameters, satisfies the cancellation property, and reduces to $G(x)$ along the$(\gamma_{1},\gamma_{2})$  diagonal. Then, it fulfills all the conditions stated in Lemma \ref{uniq}. In particular, the difference between $g_e$ and $h_e$ is zero along the diagonal. Therefore, according to Lemma \ref{uniq}, $g_e - h_e$ is always equal to zero, which implies that $g_e$ is equal to $h_e$.
\end{proof}

\begin{theorem}\label{negative_eb}
	Let $H(t)$ be a differentiable function and let $D^{-(n-m-1)}_{\gamma_{1},\gamma_{2}}(H(t))$ denote the $(n-m-1)-st$  antiderivative of $H(t)$. If $k=m-n<0$, then 
	
	\begin{align}\nonumber
		h((x_{1},w_{1}),\ldots,(x_{m},w_{m})/&(u_{1},v_{1}),\ldots,(u_{n},v_{n}))\\\label{eq:negative_eb}&=\bigg\{(n-m-1)!\left(\prod_{i=1}^{m}x_i\gamma_2(x)-w_i\gamma_1(x)\right) D^{-(n-m-1)}_{\gamma_{1},\gamma_{2}}(H(x))\bigg\}_{\gamma_{1}, \gamma_{2}}\big[\varepsilon_{1},\ldots,\varepsilon_{n}\big],
	\end{align} 
	where $(u_i,v_i)=(\gamma_{1}(\varepsilon_i),\gamma_{2}(\varepsilon_i)),\; i=1,\ldots,n$.
\end{theorem}
\begin{proof}
	We need to ensure that the right-hand side of equation (\ref{eq:negative_eb}) satisfies the four axioms of the extended blossom.
	Let
	\[
	F(x) \;=\; (n-m-1)!\,\Psi(x)\,D(x), 
	\qquad 
	\Psi(x)=\prod_{i=1}^m  \left(x_i\gamma_2(x)-w_i\gamma_1(x)\right),
	\qquad 
	D(x)=D_{\gamma_1,\gamma_2}^{-(n-m-1)}H(x).
	\]
	Then the blossom is defined by
	\[
	h((x_1,w_1),\dots,(x_m,w_m)/(u_1,v_1),\dots,(u_n,v_n))
	=\{F(x)\}_{\gamma_1,\gamma_2}[\varepsilon_1,\dots,\varepsilon_n].
	\]
	
\noindent	\textit{Bisymmetry in the $(x,w)$ and $(u,v)$ parameters:}
 $F_{\gamma_1,\gamma_2}[\varepsilon_1,\dots,\varepsilon_n]
$
is symmetric in the nodes $\varepsilon_j$ by the symmetry property of the non-polynomial divided differences. Also,
the only dependence on $(x_i,w_i)$ is through \\$
\Psi(x) = \prod_{i=1}^{m} \left(x_i\gamma_2(x)-w_i\gamma_1(x)\right),$
which is symmetric in $(x_i,w_i)$. 
Hence permuting the $(x_i,w_i)$ also leaves $\Psi(x)$ unchanged. Hence $h$ is bisymmetric.
	
	\medskip\noindent
\textit{Multilinearity in the $(x,w)$ parameters:}
The only dependence on $(x_i,w_i)$ is through $\Psi(x)$. Since each factor $x_i\gamma_2(x)-w_i\gamma_1(x)$ is linear in
the pair $(x_i,w_i)$, the product $\Psi(x)$ is separately linear in each pair.
Therefore the entire expression is multilinear in the variables $(x_i,w_i)$.
	
	\medskip\noindent
\textit{Cancellation:}
	If we insert the same parameter $(\gamma_1(\eta),\gamma_2(\eta))$
	into both the $(x,w)$--block and the $(u,v)$--block, then $\Psi(x)$
	is multiplied by $d(\eta,x)$. By the cancellation property of divided
	differences, this factor is removed when $\eta$ is also added to the list
	of divided difference nodes. Hence the value of $h$ does not change,
	so the blossom satisfies cancellation.
	
	\medskip\noindent
\textit{Diagonal property:}
	If all parameters equal to $(\gamma_1(t),\gamma_2(t))$, then
	\[
	h((\gamma_1(t),\gamma_2(t)),\dots/(\gamma_1(t),\gamma_2(t)),\dots)
	=\{F(x)\}_{\gamma_1,\gamma_2}[t,\ldots,t].
	\]
	Since $F(x)=(d(t,x))^m D(x)$, we can apply the cancellation rule exactly
	$m$ times to remove the $d(t,x)$ factors. Then we get
	\[
	(n-m-1)!\,\{D(x)\}_{\gamma_1,\gamma_2}[\underbrace{t, \ldots,t}_{n-m}].
	\]
	By the differentiation for the non-polynomial divided differences,
	\[
	\{D(x)\}_{\gamma_1,\gamma_2}[\underbrace{t, \ldots,t}_{n-m}]
	=\frac{D_{\gamma_1,\gamma_2}^{\,n-m-1}D(t)}{(n-m-1)!}.
	\]
	But $D$ was defined as the $(n-m-1)$--st antiderivative of $H$, so
	$D_{\gamma_1,\gamma_2}^{n-m-1}D(t)=H(t)$. Therefore
	\[
		h((\gamma_1(t),\gamma_2(t)),\dots,(\gamma_1(t),\gamma_2(t))/(\gamma_1(t),\gamma_2(t)),\dots,(\gamma_1(t),\gamma_2(t)))
	=(n-m-1)!\,\frac{H(t)}{(n-m-1)!}=H(t).
	\]
	
	\medskip
	Since all four axioms (bisymmetry, multilinearity, cancellation, and the diagonal
	property) are satisfied, the given formula defines the extended blossom in the case
	$k<0$. This completes the proof.
\end{proof}

\begin{theorem}\label{negativebb}
	
	\begin{align}\label{eq:negativebb}
		 H(x)_{\gamma_{1}, \gamma_{2}}\big[\varepsilon_{1},\ldots,\varepsilon_{n}\big]= (-1)^m h(\underbrace{(\delta_{1},\delta_{2}),\ldots,(\delta_{1},\delta_{2})}_{m}/(u_{1},v_{1}),\ldots,(u_{n},v_{n})), 
	\end{align}
	where $(u_i,v_i)=(\gamma_{1}(\varepsilon_i),\gamma_{2}(\varepsilon_i)),\; i=1,\ldots,n$ and $(\delta_{1},\delta_{2})=(D_{\gamma_{1}, \gamma_{2}}(\gamma_{1}(a)),D_{\gamma_{1}, \gamma_{2}}(\gamma_{2}(a)))$, $a \in \mathbb{R} $. 
\end{theorem}

\begin{proof}
Using (\ref{eq:negative_eb}) with $n=m+1$, we get
	\[
	h\big((x_1,w_1),\dots,(x_m,w_m)/(u_1,v_1),\dots,(u_n,v_n)\big)
	=\{\Psi(x)H(x)\}_{\gamma_1,\gamma_2}[\varepsilon_1,\dots,\varepsilon_n],
	\]
	where \(\Psi(x)=\prod_{i=1}^{m}(x_i\gamma_2(x)-w_i\gamma_1(x))\).
	Taking all \((x_i,w_i)=(\delta_1,\delta_2)\) with
	\((\delta_1,\delta_2)=(D_{\gamma_1,\gamma_2}\gamma_1(a),D_{\gamma_1,\gamma_2}\gamma_2(a))\),
	we obtain
	\[
	\Psi(x)=(D_{\gamma_1,\gamma_2}d(x,a))^{m}=(-1)^{m}.
	\]
	Since
	\[
	D_{\gamma_1,\gamma_2}d(x,a)
	=D_{\gamma_1,\gamma_2}[\gamma_2(a)\gamma_1(x)-\gamma_1(a)\gamma_2(x)]
	=\delta_2\gamma_1(x)-\delta_1\gamma_2(x),
	\]
	we get the result.
\end{proof}

\begin{theorem}\label{th:dif}(Differentiation)
	\begin{align}\label{diff1}
		&\binom{m}{j}g(\underbrace{(\delta_{1},\delta_{2}),\ldots,(\delta_{1},\delta_{2})}_{j},\underbrace{(\gamma_{1}(x),\gamma_{2}(x)),\ldots,(\gamma_{1}(x),\gamma_{2}(x))}_{m-j})=\frac{D_{\gamma_{1}, \gamma_{2}}^{j}G(x)}{j!},\\\label{diff2}
		&\binom{k}{j}g_{e}(\underbrace{(\delta_{1},\delta_{2}),\ldots,(\delta_{1},\delta_{2})}_{j},\underbrace{(\gamma_{1}(x),\gamma_{2}(x)),\ldots,(\gamma_{1}(x),\gamma_{2}(x))}_{m-j}/\underbrace{(\gamma_{1}(x),\gamma_{2}(x)),\ldots,(\gamma_{1}(x),\gamma_{2}(x))}_{n})=\frac{D_{\gamma_{1}, \gamma_{2}}^{j}G(x)}{j!},
	\end{align}
	where  $k=m-n\geq 0$, $(\delta_{1},\delta_{2})=(D_{\gamma_{1}, \gamma_{2}}(\gamma_{1}(a)),D_{\gamma_{1}, \gamma_{2}}(\gamma_{2}(a)))$, $a \in \mathbb{R} $.
\end{theorem}
\begin{proof}
	Proof of equation (\ref{diff1}): From  Theorem \ref{standart_b},
	\begin{align}\label{pos}
		g((x_{1},w_{1}),\ldots,(x_{m},w_{m}))=\sum_{i}\frac{(-1)^{m-i}}{m!}D_{\gamma_{1}, \gamma_{2}}^{i}\Psi(\tau) D_{\gamma_{1}, \gamma_{2}}^{m-i}G(\tau),
	\end{align}
	where
	\begin{align*} 
		\Psi(\tau)=\prod_{r=1}^{m}\left(x_r \gamma_{2}(\tau)-w_r \gamma_{1}(\tau)\right).
	\end{align*}
	For the first $j$ elements of equation (\ref{pos}), taking $(x_r,w_r)=(\delta_{1},\delta_2),\quad r=1, \ldots, j$ yields that
	\begin{align}\label{pos2}
		\Psi(\tau)=(-1)^j\prod_{r=j+1}^{m}\left(x_r \gamma_{2}(\tau)-w_r \gamma_{1}(\tau)\right).
	\end{align}
	For all $(x_r,w_r)=(\gamma_{1}(x),\gamma_{2}(x)),\quad r=j+1, \ldots, m$, we get
	\begin{align*}
		\Psi(\tau)=(-1)^j\left(d(x,\tau)\right)^{m-j}.
	\end{align*}	
	Then, for all $\tau$
	\begin{align}\label{gggg}
		g(\underbrace{(\delta_{1},\delta_{2}),\ldots,(\delta_{1},\delta_{2})}_{j},\underbrace{(\gamma_{1}(x),\gamma_{2}(x))\ldots,(\gamma_{1}(x),\gamma_{2}(x))}_{m-j})&=\sum_{i}\frac{(-1)^{m-i+j}}{m!}D_{\gamma_{1}, \gamma_{2}}^{i}\left((d(x,\tau))^{m-j}\right)D_{\gamma_{1}, \gamma_{2}}^{m-i}G(\tau).
	\end{align}
	The generalized derivative satisfies the formula (see \cite{fatma})
	\begin{align*}
		D_{\gamma_{1}, \gamma_{2}}^{k}(d(a,x))^{n}=\frac{n!}{(n-k)!}(d(a,x))^{n-k}.
	\end{align*}	
 Then using this formula, equation (\ref{gggg}) turns into
	\begin{align*}
		g(\underbrace{(\delta_{1},\delta_{2}),\ldots,(\delta_{1},\delta_{2})}_{j},\underbrace{(\gamma_{1}(x),\gamma_{2}(x)),\ldots,(\gamma_{1}(x),\gamma_{2}(x))}_{m-j})=& \sum_{i}\frac{(-1)^{m-i+j}}{m!}\frac{(m-j)!}{(m-j-i)!}(d(x,\tau))^{m-j-i}D_{\gamma_{1}, \gamma_{2}}^{m-i}G(\tau).
	\end{align*}
Now we set  $\tau = x$. Since $d(x,x)=0$, only the term with $i=m-j$ remains in the sum. Therefore,
		\begin{align*}
		g(\underbrace{(\delta_{1},\delta_{2}),\ldots,(\delta_{1},\delta_{2})}_{j},\underbrace{(\gamma_{1}(x),\gamma_{2}(x)),\ldots,(\gamma_{1}(x),\gamma_{2}(x))}_{m-j})=\frac{(m-j)!}{m!}D_{\gamma_{1}, \gamma_{2}}^{j}G(x).
	\end{align*}
Multiplying both sides by $\binom{m}{j}$ yields equation \eqref{diff1}  since $\binom{m}{j}\frac{(m-j)!}{m!}=\frac{1}{j!}$.\\

	\noindent Proof of equation (\ref{diff2}):
	\begin{align*}
		g_{e}(\underbrace{(\delta_{1},\delta_{2}),\ldots,(\delta_{1},\delta_{2})}_{j},\underbrace{(\gamma_{1}(x),\gamma_{2}(x)),\ldots,(\gamma_{1}(x),\gamma_{2}(x))}_{m-j}&/\underbrace{(\gamma_{1}(x),\gamma_{2}(x)),\ldots,(\gamma_{1}(x),\gamma_{2}(x))}_{n})\\&=g_{e}(\underbrace{(\delta_{1},\delta_{2}),\ldots,(\delta_{1},\delta_{2})}_{j},\underbrace{(\gamma_{1}(x),\gamma_{2}(x)),\ldots,(\gamma_{1}(x),\gamma_{2}(x))}_{m-n-j}/)\\&=g(\underbrace{(\delta_{1},\delta_{2}),\ldots,(\delta_{1},\delta_{2})}_{j},\underbrace{(\gamma_{1}(x),\gamma_{2}(x)),\ldots,(\gamma_{1}(x),\gamma_{2}(x))}_{m-n-j}).
	\end{align*}
	After cancellation, the expression reduces to the ordinary non-polynomial homogeneous blossom with $m-n$ arguments. Applying equation (\ref{diff1}) with 
	$m$ replaced by $m-n$ yields
	\begin{align*}
		\binom{m-n}{j}g(\underbrace{(\delta_{1},\delta_{2}),\ldots,(\delta_{1},\delta_{2})}_{j},\underbrace{(\gamma_{1}(x),\gamma_{2}(x)),\ldots,(\gamma_{1}(x),\gamma_{2}(x))}_{m-n-j})=\frac{D_{\gamma_{1}, \gamma_{2}}^{j}G(x)}{j!}.
	\end{align*}
	Therefore,
	\begin{align*}
		\binom{k}{j}g_{e}(\underbrace{(\delta_{1},\delta_{2}),\ldots,(\delta_{1},\delta_{2})}_{j},\underbrace{(\gamma_{1}(x),\gamma_{2}(x)),\ldots,(\gamma_{1}(x),\gamma_{2}(x))}_{m-j}/\underbrace{(\gamma_{1}(x),\gamma_{2}(x)),\ldots,(\gamma_{1}(x),\gamma_{2}(x))}_{n})=\frac{D_{\gamma_{1}, \gamma_{2}}^{j}G(x)}{j!}.
	\end{align*}

\end{proof}

\begin{proposition}\label{th:blossom}
	Let $g((x_{1},w_{1}),\ldots,(x_{d},w_{d}))$ denote the non-polynomial homogeneous blossom of $G(x)$. Then	
	\begin{align}\nonumber
		g_{e}((x_{1},w_{1}),&\ldots,(x_{m},w_{m})/(u_{1},v_{1}),\ldots,(u_{n},v_{n}))\\&\label{blossom1}=\frac{\sum (-1)^{\beta} g((x_{i_{1}},w_{i_{1}}),\ldots,(x_{i_{\alpha}},w_{i_{\alpha}}),(u_{j_{1}},v_{j_{1}}),\ldots,(u_{j_{\beta}},v_{j_{\beta}}))}{\binom{k}{d}}, \quad k=m-n \geq d 	\end{align}	
	\begin{align}\nonumber
		g_{e}((x_{1},w_{1}),&\ldots,(x_{m},w_{m})/(u_{1},v_{1}),\ldots,(u_{n},v_{n}))\\\label{blossom2}&=\frac{\sum (-1)^{\beta} g((x_{i_{1}},w_{i_{1}}),\ldots,(x_{i_{\alpha}},w_{i_{\alpha}}),(u_{j_{1}},v_{j_{1}}),\ldots,(u_{j_{\beta}},v_{j_{\beta}}))}{\binom{k}{d}}, \qquad k=m-n < d,
	\end{align}
	where the sum taken over all collections of indices ${i_{1},\ldots,i_{\alpha}}$ and ${j_{1},\ldots,j_{\beta}}$ such that \\
	(i) $i_{1},\ldots,i_{\alpha}$ are distinct\\
	(ii) $ j_{1},\ldots,j_{\beta}$ need not be distinct \\
	(iii) $\alpha+\beta=d$.	
\end{proposition}
\begin{proof}
	We have already proved equation (\ref{blossom1}) for the case $k = d$ in Theorem \ref{positive_eb}. The uniqueness of the extended non-polynomial homogeneous blossom makes it sufficient to demonstrate that the right-hand sides of these equations satisfy the four axioms of the extended non-polynomial homogeneous blossom. The proof for  $k\neq d$ has not changed much; the main change is that when checking the diagonal property, we need to account for the constant coefficient $ \binom{k}{d}^{-1}$. A similar proof is already given in \cite{ron1}. This proof is completed using simple counting arguments.  To obtain more comprehensive information, refer to reference \cite{ron1}.
\end{proof}

\newpage
\begin{theorem}\label{mainr}
	Let $ G(x) \in \pi_{d}(\gamma_{1}, \gamma_{2})$, and let $g^{(n-1)}$ denote the multilinear blossom of $D_{\gamma_{1}, \gamma_{2}}^{n-1}G$. Then
	\begin{align}\label{mainresult}
	G(x)_{\gamma_{1}, \gamma_{2}}\left[\varepsilon_{1},\ldots,\varepsilon_{n}\right]=\left\{ \frac{(d-n+1)!}{d!}\right\}\sum g^{(n-1)}((u_{j_{1}},v_{j_{1}}),\ldots,(u_{j_{d-n+1}},v_{j_{d-n+1}})),
	\end{align}	
	where $(u_i,v_i)=(\gamma_{1}(\varepsilon_i),\gamma_{2}(\varepsilon_i)),\; i=1,\ldots,n$ and the sum is taken over all indices $j_{1},\ldots,j_{d-n+1}$ such that $1\leq j_{1}\leq\ldots \leq j_{d-n+1}\leq n$.
\end{theorem}

\begin{proof}
The proof is modeled on the argument in \cite{ron2}.	By Theorem \ref{negativebb},
	\begin{align*}
	G(x)_{\gamma_{1}, \gamma_{2}}\left[\varepsilon_{1},\ldots,\varepsilon_{n}\right]=(-1)^{n-1} g_{e}(\underbrace{(\delta_{1},\delta_{2}),\ldots,(\delta_{1},\delta_{2})}_{n-1}/(u_{1},v_{1}),\ldots,(u_{n},v_{n})),
	\end{align*}
	and by Proposition \ref{th:blossom},
	\begin{align*}
		g_{e}(\underbrace{(\delta_{1},\delta_{2}),\ldots,(\delta_{1},\delta_{2})}_{n-1}/(u_{1},v_{1}),\ldots,(u_{n},v_{n}))=\frac{\sum (-1)^{\beta} g(\underbrace{(\delta_{1},\delta_{2}),\ldots,(\delta_{1},\delta_{2})}_{\alpha},(u_{j_{1}},v_{j_{1}}),\ldots,(u_{j_{\beta}},v_{j_{\beta}}))}{(-1)^{d}},
	\end{align*}
	where $g$ is the non-polynomial homogeneous blossom of $G$.	At this point, we note that in Proposition~\ref{th:blossom} the parameters $\alpha$ and $\beta$
	satisfy $\alpha+\beta=d$, where $\alpha$ denotes the number of arguments equal to
	$(\delta_1,\delta_2)$ and $\beta$ the number of arguments of the form $(u_j,v_j)$.
	Since the divided difference on the left-hand side is taken with respect to $n$
	nodes $\varepsilon_1,\dots,\varepsilon_n$, it represents an expression of order
	$n-1$. Consequently, on the right-hand side only those terms can contribute for
	which exactly $n-1$ arguments are equal to $(\delta_1,\delta_2)$; otherwise the
	order would not match. This forces $\alpha=n-1$, and hence
	$\beta=d-(n-1)=d-n+1$. With this choice, the sign factor simplifies as
	\[
	(-1)^{\beta}(-1)^d = (-1)^{d-n+1}(-1)^d = (-1)^{\,n-1},
	\]
	which explains the sign appearing in the resulting expression. Thus
	\begin{align}\nonumber
	&G(x)_{\gamma_{1}, \gamma_{2}}\left[\varepsilon_{1},\ldots,\varepsilon_{n}\right]\\&=(-1)^{n-1}\sum  g(\underbrace{(\delta_{1},\delta_{2}),\ldots,(\delta_{1},\delta_{2})}_{n-1},(u_{j_{1}},v_{j_{1}}),\ldots,(u_{j_{d-n+1}},v_{j_{d-n+1}}))\label{HH}.
	\end{align}
	Moreover by Theorem \ref{th:dif},
	\begin{align*}
		\frac{d!}{(d-n+1)!}g(\underbrace{(\delta_{1},\delta_{2}),\ldots,(\delta_{1},\delta_{2})}_{n-1},\underbrace{(\gamma_{1}(x),\gamma_{2}(x))\ldots,(\gamma_{1}(x),\gamma_{2}(x))}_{d-n+1})=D_{\gamma_{1}, \gamma_{2}}^{n-1}G(x).
	\end{align*}
	Therefore,
	\begin{align*}
		\frac{d!}{(d-n+1)!}g(\underbrace{(\delta_{1},\delta_{2}),\ldots,(\delta_{1},\delta_{2})}_{n-1},(u_{j_{1}},v_{j_{1}})&,\ldots,(u_{j_{d-n+1}},v_{j_{d-n+1}}))\\&=g^{(n-1)}((u_{j_{1}},v_{j_{1}}),\ldots,(u_{j_{d-n+1}},v_{j_{d-n+1}}))
	\end{align*}
	since as a function of the $x$ parameters, the left-hand side is symmetric, multiaffine and reduces to $D_{\gamma_{1}, \gamma_{2}}^{n-1}G(x)$ along the diagonal.
	Substituting this result into equation (\ref{HH}), we get equation (\ref{mainresult}).
\end{proof}
The results obtained in this paper reduce to the classical polynomial case. 
	Consider the space $\pi_n(\gamma_1,\gamma_2)$ with the choice
	\[
	\gamma_1(x)=x, \qquad \gamma_2(x)=1.
	\]
	In this case we obtain
	\[
	d(x_1,x_2)
	=\gamma_1(x_1)\gamma_2(x_2)-\gamma_1(x_2)\gamma_2(x_1)
	= x_1-x_2 .
	\]
	Consequently the generalized derivative $D_{\gamma_1,\gamma_2}$ reduces to the ordinary derivative and the non-polynomial divided difference becomes the classical polynomial divided difference. Also, the non-polynomial homogeneous blossom coincides with the standard blossom of a polynomial.  Moreover, the parameter pairs appearing in the blossom take the form $(u_i,v_i)=(\varepsilon_i,1)$.  Since in the polynomial case the blossom depends only on the first coordinate, this identity reduces exactly to the classical relation between divided differences and blossoming obtained by Goldman \cite{ron1}. Thus Theorem~\ref{mainr} may be viewed as a direct extension of Goldman's polynomial formula to the non-polynomial framework considered in this paper.

\section{Conclusions}\label{conc}
Non-polynomial divided differences satisfy numerous identities and properties. In the polynomial case,   the divided differences of a function can be expressed in terms of the blossom of the function. The main focus of this paper is to establish an analogue of that relationship. We generalize the non-polynomial homogeneous blossom  by adding a cancellation axiom and a new set of $(u,v)$ parameters to this theory. 
The non-polynomial homogeneous blossom defined in \cite{cetin2} is a special case of this general blossom. Additionally, we demonstrate the existence of a non-polynomial homogeneous variant of the multirational blossom. Using these tools, we show that there is a relation between the non-polynomial divided difference and blossoming, which is analogous to the polynomial case.

In the future, we aim to extend the non-polynomial Bernstein-B\'{e}zier bases defined in \cite{cetin2} to negative degrees and  to investigate whether the non-polynomial homogeneous variant of the multirational blossoming provides the dual functionals for the negative degree non-polynomial Bernstein-B\'{e}zier bases.

\section*{Acknowledgement}
This work was supported by Scientific Research Projects Department of Istanbul Technical
University. Project Number: THD-2023-44940.

\section*{Conflict of interest statement}

The author declares no conflicts of interest.




\end{document}